\documentclass{amsart}

\usepackage[cp1251]{inputenc}
\usepackage{amsmath}
\usepackage{amsfonts}
\usepackage{amssymb}
\usepackage{amsthm}

\usepackage[english]{babel}

\theoremstyle{plain}
\newtheorem{theoremEng}{Theorem}
\newtheorem{theoremMainEng}{Theorem}
\newtheorem{lemmaEng}[theoremEng]{Lemma}
\newtheorem{propositionEng}[theoremEng]{Proposition}

\newtheorem{corollaryLmEng}{Corollarry}[theoremEng]
\theoremstyle{definition}

\numberwithin{theoremEng}{section} \numberwithin{equation}{section}
\theoremstyle{plain}
\newtheorem*{theoremEng*}{Theorem}
\newtheorem*{lemmaEng*}{Lemma}
\newtheorem*{propositionEng*}{Proposition}
\newtheorem*{statementEng*}{Statement}
\newtheorem*{corollaryEng*}{Corollarry}
\theoremstyle{definition}
\newtheorem*{definitionEng*}{Definition}
\theoremstyle{remark}
\newtheorem*{notationEng*}{Notation}
\newtheorem*{remarkEng*}{Remark}
\newtheorem*{exampleEng*}{Example}

\renewcommand{\thetheoremMainEng}%
{\Alph{theoremMainEng}}

\title[Schr\"{o}dinger operators with measure-valued potentials]
      {Schr\"{o}dinger operators with measure-valued potentials: semiboundedness and spectrum}

\address{Institute of Mathematics of NAS of Ukraine \\
         3, Tereshchenkivska str.  \\
         Kyiv-4 \\
         01004 \\
         Ukraine}
\author[V. Mikhailets, V. Molyboga]
       {Vladimir Mikhailets, Volodymyr Molyboga}
\email[Vladimir Mikhailets]{mikhailets@imath.kiev.ua} \email[ Volodymyr Molyboga]{molyboga@imath.kiev.ua}

\dedicatory{To V. D. Koshmanenko on the occasion of his 75th birthday}

\keywords{Schr\"{o}dinger operators, strongly singular potentials, discrete spectrum, Molchanov's criterion}

\subjclass[2010]{Primary 34B20, 34L40; Secondary 34B24, 47A10}

\begin{document}
\begin{abstract}
We study the 1-D Schr\"{o}dinger operators in Hilbert space $L^{2}(\mathbb{R})$ with real-valued Radon measure  $q'(x)$, $q\in \mathrm{BV}_{loc}(\mathbb{R})$ as potentials. 
New sufficient conditions for minimal operators to be bounded below and selfadjoint are found. 
For such operators a criterion for the discreteness of the spectrum is proved, which generalizes Molchanov's, Brinck's, and the Albeverio--Kostenko--Malamud criteria. The quadratic forms corresponding to the investigated operators are described.
\end{abstract}

\maketitle

\section{Introduction and main results}\label{sec:VtaOR}
We consider the 1-D Schr\"{o}dinger operator 
\begin{equation}\label{eq_VtaOR10}
	 \mathrm{S}(q)u\equiv \mathrm{S}u:=-u''+q'(x)u,
\end{equation}
in the complex Hilbert space $L^{2}(\mathbb{R})$. 
The potential of \eqref{eq_VtaOR10} is the generalized derivative $q'(x)$ 
of a certain real-valued function $q\in L_{loc}^{2}(\mathbb{R})$. 
Following \cite{SvSh2003}, we define $\mathrm{S}(q)$ as a quasi-differential operator 
\begin{align*}\label{eq_VtaOR12}
l_{q}[u] & :=-(u'-qu)'-q(u'-qu)-q^{2}u, \\
\mathrm{Dom}(l_{q}) & :=\left\{u:\mathbb{R}\rightarrow \mathbb{C}\left|\,u,u'-qu\in \mathrm{AC}_{loc}(\mathbb{R})\right.\right\}.
\end{align*}
The quasi-differential expression $l_{q}[u]$ is equal to $-u''+q'(x)u$ in the sense of distributions 
\begin{equation*}\label{eq_VtaOR14}
\langle l_{q}[u],\varphi \rangle = \langle -u''+q'(x)u,\varphi \rangle\qquad \text{for every } \varphi\in C_{comp}^{\infty}(\mathbb{R}).
\end{equation*}
Hereafter $u^{[1]}:=u'-qu$ denotes the quasi-derivative. Then the operators \eqref{eq_VtaOR10} are defined as 
\begin{align*}
\mathrm{S}(q)u & :=l_{q}[u],   \\
\mathrm{Dom}(\mathrm{S}(q)) & :=\left\{u\in L^{2}(\mathbb{R})\left|\,u,u'-qu\in \mathrm{AC}_{loc}(\mathbb{R}),\,l_{q}[u]\in L^{2}(\mathbb{R})\right.\right\},\notag\hspace{75pt}
\end{align*}
and
\begin{equation*}
\dot{\mathrm{S}}_{0}(q)u :=l_{q}[u], \qquad
\mathrm{Dom}(\dot{\mathrm{S}}_{0}(q)) :=\left\{u\in \mathrm{Dom}(\mathrm{S}(q))\left|\,\mathrm{supp}\,u\Subset\mathbb{R}\right.\right\}. 
\end{equation*}
As usual the operators $\mathrm{S}(q)$ and $\dot{\mathrm{S}}_{0}(q)$ are called maximal and preminimal respectively. 
Under these assumptions the operator $\dot{\mathrm{S}}_{0}(q)$ is symmetric and closable, 
its closure being denoted by $\mathrm{S}_{0}(q)$ (see Proposition in Appendix).

Necessary and sufficient conditions for the operators $\mathrm{S}_{0}(q)$ to be bounded below and to have discrete spectrum are found in \cite{MkMrNvMFAT2017}. 
However, they are not constructive. 
Nonetheless, in physical applications the most interesting situation is where the potentials  $q'(x)$ in \eqref{eq_VtaOR10} are real-valued Radon measures on a locally compact space $\mathbb{R}$, i.~e. $q\in \mathrm{BV}_{loc}(\mathbb{R})$ (see, for instance, references in \cite{Bra1985,AlKsMl2010,KshDdk2016}). 
This situation is investigated in this paper. 
The case where Radon measure is absolutely continuous, i.~e. $q'\in L_{loc}^{1}(\mathbb{R})$, was studied in \cite{Bri1959,Mol1953}. 
The approach applied in \cite{Bri1959} may be generalized onto arbitrary Radon measures on $\mathbb{R}$. 

Let us suppose that there exists a finite number $C>0$ such that for all intervals $J$ 
of the real axis $\mathbb{R}$ with length $\leq 1$ we have 
\begin{equation*}\label{eq_VtaOR34.1}
\int_{J}d\,q(x)\geq -C. \tag{$\mathrm{Br}$}
\end{equation*}
Without loss of generality we may assume that in the Brinck condition \eqref{eq_VtaOR34.1} $C\geq 2$ and we assume this in what follows.

\begin{theoremMainEng}\label{thM_VtaOR_B}
Under the condition $(\mathrm{Br})$ the operator $\mathrm{S}_{0}(q)$ is bounded below, selfadjoint and $\mathrm{S}_{0}(q)=\mathrm{S}(q)$.
\end{theoremMainEng}
The following theorem gives necessary and sufficient conditions for the spectra of the minimal operators to be discrete.
\begin{theoremMainEng}\label{thM_VtaOR_C}
	Let the potential $q'(x)$ satisfy the condition $(\mathrm{Br})$. Then spectrum of the operator $\mathrm{S}_{0}(q)$ is discrete if and only if the Molchanov condition is satisfied 
	\begin{equation*}\label{eq_VtaOR40}
	\lim_{|a|\rightarrow\infty}\int_{a}^{a+h}d\,q(x)=+\infty
	\end{equation*}
	for all $h>0$.
\end{theoremMainEng}
The following theorem gives a description of the quadratic forms generated by the Schr\"{o}dinger operators. We use notations and definitions from~\cite{Kt1995}.
\begin{theoremMainEng}\label{thM_VtaOR_D}
	Let the potential $q'(x)$ satisfy the condition $(\mathrm{Br})$. Then following statements are fulfilled.
	\begin{itemize}
		\item [(I)] The sesquilinear form 
		\begin{align*}\label{eq_VtaOR44}
		\dot{t}_{\dot{\mathrm{S}}_{0}(q)}[u,v]\equiv \dot{t}[u,v] & :=\left(\dot{\mathrm{S}}_{0}(q)u,v\right)_{L^{2}(\mathbb{R})}
		=\int_{\mathbb{R}}u'\overline{v'}d\,x+\int_{\mathbb{R}}u\overline{v}d\,q(x), \\ 
		\mathrm{Dom}(\dot{t}_{\dot{\mathrm{S}}_{0}(q)}) & :=\mathrm{Dom}(\dot{\mathrm{S}}_{0}(q)),
		\end{align*}
		is densely defined, symmetric, and bounded below 
		\begin{equation*}
		  \left(\dot{\mathrm{S}}_{0}(q)u,u\right)_{L^{2}(\mathbb{R})}\geq -2C^{2} \lVert u\rVert_{L^{2}(\mathbb{R})}^{2}.
		\end{equation*}
		The form $\dot{t}_{\dot{\mathrm{S}}_{0}(q)}$ is closable.
	    \item [(II)] "Potential energy"
        \begin{equation*}\label{eq_VtaOR42}
          \mathrm{Q}(u):=\lim_{M,N\rightarrow\infty}\int_{-M}^{N}|u(x)|^{2}d\,q(x)
        \end{equation*}
         exists and is finite for all $u\in \mathrm{Dom}(\mathrm{S}(q))$, moreover 
         \begin{equation*}
           \mathrm{Dom}(\mathrm{S}(q))\subset H^{1}(\mathbb{R}).
         \end{equation*}	
		\item [(III)] The closure $t$ of the sesquilinear form $\dot{t}$, $t:=(\dot{t})^{\sim}$, may be represented as:
		\begin{align*}
		t[u,v] & =\int_{\mathbb{R}}u'\overline{v'}d\,x+\lim_{M,N\rightarrow\infty}\int_{-M}^{N}u\overline{v}d\,q(x),  \\
		\mathrm{Dom}(t) &  =\left\{u\in H^{1}(\mathbb{R})\left|\,\,\exists \,\lim_{M,N\rightarrow\infty}\int_{-M}^{N}|u(x)|^{2}d\,q(x)\in \mathbb{R}\right.\right\}.
		\end{align*}
		The sesquilinear form $t$ is densely defined, closed, symmetric, and bounded below.
	\end{itemize}
\end{theoremMainEng}


\section{Proof of Theorem \ref{thM_VtaOR_B}}
We begin with formulating two necessary lemmas.

\begin{lemmaEng}[T.~Ganelius \cite{Gn1956}]\label{lm_PRtaZ10}
	Let $f\geq 0$ and $g$ be functions of bounded variation on a compact interval $J$. Then
	\begin{equation*}
	\int_{J}fd\,g\leq\left(\inf_{J}f+\mathop\mathrm{var}\limits_{J} f\right)\sup_{K\subset J}\int_{K}d\,g,
	\end{equation*}
	where $K$ is a compact subinterval of $J$. 
\end{lemmaEng}
Lemma \ref{lm_PRtaZ10} is crucial in our proof of the fact that the preminimal operator~$\dot{\mathrm{S}}_{0}(q)$ is  bounded below under the condition $(\mathrm{Br})$.

The following lemma plays a technical role.
\begin{lemmaEng}[I. Brinck \cite{Bri1959}]\label{lm_PRtaZ12}
	Let $J$ be a compact interval of length $l$. Then for all $x\in J$ and $f\in H^{1}(J)$ we have
	\begin{equation*}\label{eq_PRtaZ12}
	\frac{1}{2}l^{-1}\| f\|_{L^{2}(J)}^{2}-\frac{1}{2}l\| f'\|_{L^{2}(J)}^{2}\leq |f(x)|^{2}\leq 2t^{-1}\| f\|_{L^{2}(J)}^{2}+ t\| f'\|_{L^{2}(J)}^{2},\quad 0<t\leq l,
	\end{equation*}
	and
	\begin{equation*}\label{eq_PRtaZ14}
	\inf_{x\in J}|f(x)|^{2}\leq l^{-1}\| f\|_{L^{2}(J)}^{2}.
	\end{equation*}
\end{lemmaEng}

\begin{lemmaEng}\label{lm_PRtaZ14}
Let $q'(x)$ satisfy the condition $(\mathrm{Br})$. 
If $I$ is a finite interval of length $l$ and if $f\in H^{1}(I)$, then 
\begin{equation}\label{eq_PRtaZ16}
 \int_{I}|f(x)|^{2}d\,q(x)\geq -C\left(2(hl/n)^{-1}\| f\|_{L^{2}(I)}^{2}+(hl/n)\| f'\|_{L^{2}(I)}^{2}\right),
\end{equation}
where $n$ is an integer such that $n-1<l\leq n$, and $h$ is an arbitrary number from $(0,1]$.
\end{lemmaEng}
\begin{proof}
There is no loss of generality in supposing that $I=(0,l)$. 

We first suppose $l=1$ and apply Lemma~\ref{lm_PRtaZ10}. Thus 
\begin{equation*}\label{eq_PRtaZ18}
-\int_{I}|f(x)|^{2}d\,q(x)\leq -\left(\inf_{I}\lvert f\rvert^{2}+\mathop\mathrm{var}\limits_{J} \lvert f\rvert^{2}\right)\sup_{K\subset I}\int_{K}d\,q(x).
\end{equation*}
Due to $(\mathrm{Br})$ the factor $-\sup_{K\subset I}\int_{K}d\,q(x)$ is majorized by $C$, and from Lemma~\ref{lm_PRtaZ12} we get
\begin{equation*}
\inf_{x\in I}|f(x)|^{2}\leq \| f\|_{L^{2}(I)}^{2}\leq h^{-1} \| f\|_{L^{2}(I)}^{2},\qquad h\in (0,1].
\end{equation*}
We now write $f(x)=f_{1}(x)+if_{2}(x)$, where $f_{1}$ and $f_{2}$ are real functions. 
Due to Cauchy's inequality we get 
\begin{align*}\label{eq_PRtaZ20}
\mathop\mathrm{var}_{I}|f(x)|^{2} & =\int_{I}\left|\frac{d}{dx}|f(x)|^{2}\right|dx= \int_{I}\left|f_{1}f_{1}'+f_{2}f_{2}'\right|dx
\leq 2\| f\|_{L^{2}(I)}\| f'\|_{L^{2}(I)},
\end{align*}
and, hence, 
\begin{equation*}
\begin{split}
-\int_{I}|f(x)|^{2}d\,q(x)\leq C h^{-1}\left(\| f\|_{L^{2}(I)}^{2}+2h\| f\|_{L^{2}(I)}\| f'\|_{L^{2}(I)}\right)\\\leq
Ch^{-1}\left(2\| f\|_{L^{2}(I)}^{2}+h^{2}\| f'\|_{L^{2}(I)}^{2}\right),
\end{split}
\end{equation*}
which proves the lemma for $l=1$.

To prove the lemma for arbitrary $l$ we put $Q(x)=q(ln^{-1}x)$. Then
\begin{align*}\label{eq_PRtaZ22}
\int_{0}^{l}|f(x)|^{2}d\,q(x) & =\int_{0}^{n}|f(ln^{-1}x)|^{2}d\,q(ln^{-1}x)=\int_{0}^{n}|f(ln^{-1}x)|^{2}d\,Q(x) \\
& =\sum_{k=1}^{n}\int_{k-1}^{k}|f(ln^{-1}x)|^{2}d\,Q(x).
\end{align*}
Note that the function $Q$ satisfies condition $(\mathrm{Br})$ with the same constant $C$ for all intervals of
length $\leq n/l$ and, hence, for all intervals of length $\leq 1$. 
Therefore the assumption of lemma for intervals of unit length implies 
\begin{equation*}\label{eq_PRtaZ24}
\int_{k-1}^{k}|f(ln^{-1}x)|^{2}d\,Q(x)\geq -C\left(2h^{-1}\int_{k-1}^{k}|f(ln^{-1}x)|^{2}d\,x+h\int_{k-1}^{k}\frac{d}{dx}|f(ln^{-1}x)|^{2}d\,x\right),
\end{equation*}
and hence, summing over $k$, we get
\begin{align*}
\int_{0}^{n}|f(ln^{-1}x)|^{2}d\,Q(x) & \geq -C\left(2h^{-1}\int_{0}^{n}|f(ln^{-1}x)|^{2}d\,x+ h\int_{0}^{n}\frac{d}{dx}|f(ln^{-1}x)|^{2}d\,x\right) \\
& =-C\left(2h^{-1}l^{-1}n\int_{0}^{l}|f(x)|^{2}d\,x+ h l n^{-1}\int_{0}^{l}|f'(x)|^{2}d\,x\right),
\end{align*}
which proves the lemma.
\end{proof}

\begin{corollaryLmEng}\label{cor_10}
If the length of an interval $I$ does not exceed $1$, then 
\begin{equation*}
	\int_{I}|u'(x)|^{2}d\,x+2C^{2}\int_{I}|u(x)|^{2}d\,x+\int_{I}|u(x)|^{2}d\,q(x)\geq 0
\end{equation*}
for any $u\in H^{1}(I)$.
\end{corollaryLmEng}
\begin{proof}
Due to the choice of $n$ in Lemma~\ref{lm_PRtaZ14}, we get $n/lC<(l+1)/lC$. 
Since we assume that $C\geq 2$, we may conclude that $n/lC<1$ if $l\geq 1$. 
Thus, we may put $h=n/lC$ in \eqref{eq_PRtaZ16}, which yields the corollary.
\end{proof}
\begin{corollaryLmEng}\label{cr_10LmPRtaZ14}
Let the condition $(\mathrm{Br})$ be satisfied. Then 
\begin{equation}\label{eq_PRtaZ25}
	\int_{\mathbb{R}}|u(x)|^{2}d\,q(x)\geq -C\left(2h^{-1}\| u\|_{L^{2}(\mathbb{R})}^{2}+h\| u'\|_{L^{2}(\mathbb{R})}^{2}\right)\quad
\end{equation}
for all $ u\in H_{comp}^{1}(\mathbb{R})$ and $h\in (0,1].$
\end{corollaryLmEng}
\begin{proof}
We divide the real axis into a sum of disjoint intervals of unit length. 
Then \eqref{eq_PRtaZ16} holds on each of these intervals and the summation gives \eqref{eq_PRtaZ25}.
\end{proof}
\begin{remarkEng*}
If the support of $u$ is not compact, corollary \ref{cr_10LmPRtaZ14} obviously still holds if
\begin{equation*}
	\lim_{M,N\rightarrow\infty}\int_{-M}^{N}|u(x)|^{2}d\,q(x)
\end{equation*}
exists as improper Riemann--Stieltjes integral. 
Then the integral in \eqref{eq_PRtaZ25} must, of course, be interpreted accordingly.
\end{remarkEng*}

Lemma \ref{lm_PRtaZ14} allows us to prove that the preminimal operator is bounded below.
\begin{theoremEng}\label{th_PRtaZ10}
	Let the potential $q'(x)$ satisfy the condition $(\mathrm{Br})$. Then the preminimal operator $\dot{\mathrm{S}}_{0}(q)$ is bounded below and the following estimate holds:
	\begin{equation*}\label{eq_PRtaZ26}
	(\dot{\mathrm{S}}_{0}(q)u,u)\geq -2 C^{2}\|u\|_{L^{2}(\mathbb{R})}^{2},\qquad u\in \mathrm{Dom}(\dot{S}_{0}(q)).
	\end{equation*}
\end{theoremEng}
\begin{proof}
	For arbitrary $u\in \mathrm{Dom}(\dot{\mathrm{S}}_{0}(q))$ there is a positive integer $N$ such that  $\mathrm{supp}\,u\subseteq [-N,N]$ (recall that 
	$\mathrm{Dom}(\dot{\mathrm{S}}_{0}(q))\subset H_{comp}^{2}(\mathbb{R})$, see property~$6^{0}$ of Proposition in Appendix). Therefore 
	\begin{align}\label{eq_PRtaZ28}
	(\dot{\mathrm{S}}_{0}(q)u,u)_{L^{2}(\mathbb{R})} & =(l_{q}[u],u)_{L^{2}(\mathbb{R})}=\|u'\|_{L^{2}(\mathbb{R})}^{2}+\int_{\mathbb{R}}|u(x)|^{2}d\,q(x) \\
	& =\|u'\|_{L^{2}(\mathbb{R})}^{2}+\sum_{n=-N}^{N}\int_{[n,n+1)}|u(x)|^{2}d\,q(x). \notag
	\end{align}
	To estimate terms $\int_{[n,n+1)}|u(x)|^{2}d\,q(x)$ we apply Lemma~\ref{lm_PRtaZ14} with $l=n=1$ and $h=C^{-1}$ (recall that $C\geq 2$) and get 
	\begin{equation}\label{eq_PRtaZ30}
	\int_{[n,n+1)}|u(x)|^{2}d\,q(x)\geq -2 C^{2}\|u\|_{L^{2}([n,n+1))}^{2}-\|u'\|_{L^{2}([n,n+1))}^{2}.
	\end{equation}
	
	Substituting the estimate \eqref{eq_PRtaZ30} into \eqref{eq_PRtaZ28} we receive the estimate we require:
	\begin{equation*}\label{eq_PRtaZ32}
	\begin{split}
	(\dot{\mathrm{S}}_{0}(q)u,u)_{L^{2}(\mathbb{R})}\geq  \|u'\|_{L^{2}(\mathbb{R})}^{2}
	+\sum_{n=-N}^{N}\left(-2 C^{2}\|u\|_{L^{2}([n,n+1))}^{2}-\|u'\|_{L^{2}([n,n+1))}^{2}\right)=\\
	=-2 C^{2}\|u\|_{L^{2}(\mathbb{R})}^{2}.
	\end{split}
	\end{equation*}
	
	Theorem is proved.
\end{proof}

If the preminimal operator $\dot{\mathrm{S}}_{0}(q)$ is bounded below, then the minimal operator $\mathrm{S}_{0}(q)$ is selfadjoint and coincides with the maximal operator $\mathrm{S}(q)$ (see \cite[Remark~III.2]{AlKsMl2010} and \cite[Corollary 2]{MkMlMFAT2013n2}). 
Therefore Theorem~\ref{th_PRtaZ10} implies Theorem~\ref{thM_VtaOR_B}.

Theorem~\ref{thM_VtaOR_B} is proved. 

\section{Auxiliary results}
We shall make use of a set of functions $\varphi(x)$ with compact supports and uniformly bounded derivatives. We define
$\varphi$ as follows:
\begin{equation}\label{eq_33}
\begin{array}{cl}
\text{(i)} & \varphi(x)=\varphi(x,r,R)=
\begin{cases}
1 & \text{for}\quad -r\leq x\leq R, \\
0 & \text{for}\quad x<-r-1\;\text{and}\; x>R+1.
\end{cases}
\\
\text{(ii)}& \text{For every}\; x \text{ the function } \varphi(x)\;\text{is increasing in}\;r\;\text{and}\;R. \\
\text{(iii)}&\text{The derivatives}\;\varphi'(x)\;\text{and}\;\varphi''(x)\;\text{are continuous and uniformly bounded}\\ &\text{in}\;x, r\;\text{and}\;R.
\end{array}
\end{equation}

It follows from this definition that $0\leq \varphi\leq 1$ and that $\varphi\rightarrow 1$ as $\min(r,R)\rightarrow\infty$.

\begin{lemmaEng}\label{l_20}
Let $\omega:\mathbb{R}\rightarrow\mathbb{R}$ be a bounded, twice continuously differentiable function with bounded first and second derivatives. If
\begin{equation}\label{eq_34}
	\int_{J}\omega(x)d\,q(x)\geq -C
\end{equation}
for all intervals $J$ of length $\leq 1$, then
\begin{equation*}
	\int_{\mathbb{R}}\omega^{2}|u'|^{2}d\,x< \infty
\end{equation*}
for all $u\in \mathrm{Dom}(\mathrm{S}(q))$.
\end{lemmaEng}
\begin{proof}
Let $\varphi$ be one of the functions introduced above and put 	$\psi=\varphi^{2}\omega^{2}$. If $u$ is any function in $\mathrm{Dom}(\mathrm{S}(q))$ we get, integrating by parts,
\begin{equation}\label{eq_36}
	\int_{\mathbb{R}}\psi l_{q}[u]\overline{u}d\,x=\int_{\mathbb{R}}\psi'u'\overline{u}d\,x+
	\int_{\mathbb{R}}\psi|u'|^{2}d\,x+\int_{\mathbb{R}}\psi|u|^{2}d\,q.
\end{equation}
Now, let $u$ be a real function in $\mathrm{Dom}(\mathrm{S}(q))$. Then the first integral on the right can be integrated by parts, yielding
\begin{equation}\label{eq_38}
	\int_{\mathbb{R}}\psi l_{q}[u]\overline{u}d\,x=-\frac{1}{2}\int_{\mathbb{R}}\psi''|u|^{2}d\,x+
	\int_{\mathbb{R}}\psi|u'|^{2}d\,x+\int_{\mathbb{R}}\psi|u|^{2}d\,q.
\end{equation}
The functions $\psi$ and $\psi''$ tend boundedly to $\omega^{2}$ and $(\omega^{2})''$ respectively as $\varphi\rightarrow 1$, that is as $\min(r,R)\rightarrow\infty$, and, since $|\omega^{2} l_{q}[u]\overline{u}|$ and $(\omega^{2})''|u|^{2}$ are both integrable, the first two integrals in \eqref{eq_38} tend to the finite limits $\int_{\mathbb{R}}\omega^{2} l_{q}[u]\overline{u}d\,x$ and $\int_{\mathbb{R}}(\omega^{2})''|u|^{2}d\,x$ respectively as $\varphi\rightarrow 1$. Since the convergence of $\psi$ is also monotone, we conclude that $\int_{\mathbb{R}}\psi|u'|^{2}d\,x$ must tend to $\int_{\mathbb{R}}\omega^{2}|u'|^{2}d\,x$ although this limit may not be finite, and therefore $\int_{\mathbb{R}}\psi|u|^{2}d\,q$ must also have limit (possibly $-\infty$).
	
We put $d\,W(x)=\omega(x)d\,q(x)$. It follows from \eqref{eq_34} that $W$ satisfies a condition of the type $(\mathrm{Br})$. 
Therefore, we apply Lemma~\ref{lm_PRtaZ10} (as in the proof of Lemma~\ref{lm_PRtaZ14}) to obtain
\begin{equation*}
	-\int_{\mathbb{R}}\psi|u|^{2}d\,q =-\int_{\mathbb{R}}\omega(x)\varphi^{2}(x)|u(x)|^{2}d\,W(x)
	\leq C\left(2\int_{\mathbb{R}}\varphi^{2}\omega|u|^{2}d\,x
	+\mathrm{var}\varphi^{2}\omega|u|^{2}\right).
\end{equation*}
But $\mathrm{var}\varphi^{2}\omega|u|^{2}$ is bounded by
\begin{equation*}
	\int_{\mathbb{R}}\varphi^{2}|\omega'||u|^{2}d\,x+2\int_{\mathbb{R}}\varphi\omega|\varphi'||u|^{2}d\,x+
	2\int_{\mathbb{R}}\omega\varphi^{2}|uu'|d\,x,
\end{equation*}
which in turn is majorized by
\begin{equation*}
	M\|u\|^{2}+2\|u\|\|\varphi\omega u'\|,
\end{equation*}
where the coefficient $M$ depends only on the bounds for $\omega$,	$\omega'$, and $\varphi'$. 
Hence, it follows from \eqref{eq_38} that 
\begin{equation*}
	\|\varphi\omega u'\|^{2}\leq O(1)+2\|u\|\|\varphi\omega u'\|.
\end{equation*}
Thus, $\|\varphi\omega u'\|^{2}=\int_{\mathbb{R}}\varphi^{2}\omega^{2} |u'|^{2}d\,x=\int_{\mathbb{R}}\psi|u'|^{2}d\,x$ must be bounded. Therefore
\begin{equation}\label{eq_40}
	\int_{\mathbb{R}}\omega^{2}|u'|^{2}d\,x<\infty,
\end{equation}
and the lemma is proved for every real $u\in \mathrm{Dom}(\mathrm{S}(q))$.
	
Since every $u$ in $\mathrm{Dom}(\mathrm{S}(q))$ may be written $u_{1}+iu_{2}$, where $u_{1}$ and $u_{2}$ are real and from $\mathrm{Dom}(\mathrm{S}(q))$, the proof for real $u$ shows that $\int_{\mathbb{R}}\omega^{2}|u_{1}'|^{2}d\,x<\infty$ and $\int_{\mathbb{R}}\omega^{2}|u_{2}'|^{2}d\,x<\infty$. Hence, $\int_{\mathbb{R}}\omega^{2}|u'|^{2}d\,x<\infty$ for all $u\in \mathrm{Dom}(\mathrm{S}(q))$. 
The proof of the lemma is complete.
\end{proof}

We observe that $\int_{\mathbb{R}}\psi|u|^{2}d\,q$ has a finite limit for all $u$ in $\mathrm{Dom}(\mathrm{S}(q))$, and that $|u'\overline{u}|$ is integrable. 
Hence $\int_{\mathbb{R}}\psi'u'\overline{u}d\,x$ in \eqref{eq_36} tends
to $\int_{\mathbb{R}}(\omega^{2})'u'\overline{u}d\,x$ for all $u\in \mathrm{Dom}(\mathrm{S}(q))$.

We obtain the following useful result from Lemma \ref{l_20} with $\omega(x)\equiv 1$.
\begin{corollaryLmEng}\label{cr_dom}
Let the condition $(\mathrm{Br})$ be satisfied. Then 
\begin{equation*}
 \mathrm{Dom}(\mathrm{S}(q))\subset H^{1}(\mathbb{R}).
\end{equation*}
\end{corollaryLmEng}	

We see from \eqref{eq_36}, \eqref{eq_38} and \eqref{eq_40} with $\omega(x)\equiv 1$ that $\|u'\|^{2}$ is finite and that
\begin{equation*}
\lim_{\varphi\rightarrow 1}\int_{\mathbb{R}}\varphi^{2}|u|^{2}d\,q(x)\; \text{exists}
\end{equation*}
and also that
\begin{equation*}
(\mathrm{S}(q) u,u)_{L^{2}(\mathbb{R})}=\int_{\mathbb{R}}l_{q}[u]\overline{u}d\,x=\int_{\mathbb{R}}|u|^{2}d\,x+\lim_{\varphi\rightarrow 1}\int_{\mathbb{R}}\varphi^{2}|u|^{2}d\,q(x).
\end{equation*}
This enables us to prove that the ''potential energy''
\begin{equation}\label{eq_42}
\mathrm{Q}(u)=\lim_{M,N\rightarrow\infty}\int_{-M}^{N}|u(x)|^{2}d\,q(x)
\end{equation}
exists and is finite for every $u\in \mathrm{Dom}(\mathrm{S}(q))$ as improper Riemann--Stieltjes integral.

Let $\varphi_{1}=\varphi^{2}(x,r,R)$ and $\varphi_{2}=\varphi^{2}(x,r-1,R-1)$, 
with $\varphi$ being defined by \eqref{eq_33}.
Then obviously
\begin{equation*}
\int_{-r}^{R}|u|^{2}d\,q=\int_{\mathbb{R}}\varphi_{1}|u|^{2}d\,q
-\int_{-r-1}^{-r}\varphi_{1}|u|^{2}d\,q-\int_{R}^{R+1}\varphi_{1}|u|^{2}d\,q
\end{equation*}
and
\begin{equation*}
\int_{-r}^{R}|u|^{2}d\,q=\int_{\mathbb{R}}\varphi_{2}|u|^{2}d\,q
+\int_{-r}^{-r+1}(1-\varphi_{2})|u|^{2}d\,q+\int_{R-1}^{R}(1-\varphi_{2})|u|^{2}d\,q.
\end{equation*}
In these two identities four integrals over intervals of unit length can each be one-sidedly estimated 
by the norms of $u$ and $u'$ over the interval by Lemma~\ref{lm_PRtaZ14}. 
Since $u$ and $u'$ are both from $L^{2}(\mathbb{R})$, those norms vanish with increasing
$r$ and $R$. Thus
\begin{equation*}
\int_{\mathbb{R}}\varphi_{2}|u|^{2}d\,q-o(1)\leq \int_{-r}^{R}|u|^{2}d\,q\leq
\int_{\mathbb{R}}\varphi_{1}|u|^{2}d\,q+o(1),
\end{equation*}
and, hence, 
\begin{equation*}
\int_{-r}^{R}|u|^{2}d\,q\rightarrow \lim_{\varphi\rightarrow 1}
\int_{\mathbb{R}}\varphi|u|^{2}d\,q
\end{equation*}
as $\min(r,R)\rightarrow\infty$. Thus, the limit in \eqref{eq_42} exists. It also follows that
\begin{equation}\label{eq_44}
\mathrm{Q}(u)=\int_{\mathbb{R}}|u|^{2}d\,q=\left(\mathrm{S}(q)u,u\right)_{L^{2}(\mathbb{R})}-(u',u')_{L^{2}(\mathbb{R})}
\end{equation}
for all $u\in \mathrm{Dom}(\mathrm{S}(q))$, which is equivalent to
\begin{equation*}
\left(\mathrm{S}(q) u,u\right)_{L^{2}(\mathbb{R})}=\int_{\mathbb{R}}|u'|^{2}d\,x+\int_{\mathbb{R}}|u|^{2}d\,q(x).
\end{equation*}
We have just proved the first half of the following
\begin{theoremEng}\label{th_16}
If $q'(x)$ satisfies $(\mathrm{Br})$, then the potential energy $\mathrm{Q}(u)$ defined by \eqref{eq_42} 
exists and is finite for any $u\in \mathrm{Dom}(\mathrm{S}(q))$ as improper Riemann--Stieltjes integral. 
Moreover, for any $h\in (0,C^{-1}]$ 
and every $u\in\mathrm{Dom}(\mathrm{S}(q))$, we have
\begin{equation}\label{eq_46}
	(1-C h)(u',u')_{L^{2}(\mathbb{R})}\leq 2C h^{-1}(u,u)_{L^{2}(\mathbb{R})}+\left(\mathrm{S}(q) u,u\right)_{L^{2}(\mathbb{R})}
\end{equation}
and
\begin{equation}\label{eq_48}
	(1-C h)\mathrm{Q}(u)\geq -2C h^{-1}(u,u)_{L^{2}(\mathbb{R})}-C h\left(\mathrm{S}(q) u,u\right)_{L^{2}(\mathbb{R})}.
\end{equation}
\end{theoremEng}
\begin{proof}
For all $h\leq C^{-1}$ $(<1)$ and every $u\in\mathrm{Dom}(\mathrm{S}(q))$ we have 
\begin{equation*}
	\mathrm{Q}(u)=\int_{\mathbb{R}}|u|^{2}d\,q(x)\geq -2Ch^{-1}\|u\|_{L^{2}(\mathbb{R})}^{2}-Ch\|u'\|_{L^{2}(\mathbb{R})}^{2}
\end{equation*}
due to Corollary \ref{cr_10LmPRtaZ14} and the remark to this Corollary. 
Then \eqref{eq_46} and \eqref{eq_48} follow from \eqref{eq_44}.
\end{proof}


\section{Proof of Theorem \ref{thM_VtaOR_C}}
Let us first prove some preliminary results.

If $q'(x)$ satisfies an upper estimate of a type corresponding to $(\mathrm{Br})$, that is
\begin{equation}\label{eq_26}
\int_{J}d\,q(x)\leq C_{1}
\end{equation}
for all intervals $J$ of length $\leq 1$, then $-q'(x)$ satisfies $(\mathrm{Br})$ with $C$ replaced by $C_{1}$. 
Hence, Lemma~\ref{lm_PRtaZ14} and Corollary~\ref{cr_10LmPRtaZ14} give upper bounds for
$\int|u|^{2}d\,q$. For convenience we state them in a separate statement.
\begin{propositionEng}\label{pr_12}
Let $q'(x)$ satisfy \eqref{eq_26}. If $I$ is any finite interval of length $l$ and $f\in H_{2}^{1}(I)$, then 
\begin{equation*}
	\int_{I}|u(x)|^{2}d\,q(x)\leq C_{1}\left\{2(hl/n)^{-1}\parallel u\parallel_{L^{2}(I)}^{2}+(hl/n)\parallel
	f'\parallel_{L^{2}(I)}^{2}\right\},
\end{equation*}
where $n$ is the integer determined by $n-1<l\leq n$ and $h$ is any number in the interval $0<h\leq 1$.
	
If $u$ belongs to $H^{1}(\mathbb{R})$ and has compact support, then	also
\begin{equation*}
	\int_{\mathbb{R}}|u(x)|^{2}d\,q(x)\leq C_{1}\left\{2h^{-1}
	\parallel u\parallel_{L^{2}(\mathbb{R})}^{2}+h\parallel u'\parallel_{L^{2}(\mathbb{R})}^{2}\right\}
\end{equation*}
for any positive $h\leq 1$.
\end{propositionEng}

\begin{lemmaEng}\label{lm_DOR12}
Assume that $I$ is an interval of length $\leq 1$, $q'(x)$ satisfies $(\mathrm{Br})$ and 
\begin{equation*}
	\int_{I}|h(x)|^{2}d\,q(x)\leq C_{1}
\end{equation*}
for some function $h\in H^{1}(I)$ such that
	$0<m\leq |h(x)|\leq M$
for all $x\in I$. Then
\begin{equation}\label{eq_DORLm10}
	\int_{I}d\,q(x)\leq C_{0},
\end{equation}
where $C_{0}$ depends only on $C$, $C_1$, $m$, $M$, and $\lVert h'\rVert_{L^{2}(I)}^{2}$.
\end{lemmaEng}
\begin{proof}
We apply Lemma \ref{lm_PRtaZ10} with $f=|h|^{-2}$ and $d\,g=|h|^{2}d\,q$ to obtain
\begin{equation}\label{eq_DORLm11}
\int_{I}d\,q(x)=\int_{I}|h|^{-2}|h|^{2}d\,q(x)\leq \left(\inf_{I}|h|^{-2}+\mathop\mathrm{var}_{I}|h|^{-2}\right)\sup_{J\subset I} \int_{J}|h|^{2}d\,q(x),
\end{equation}
and we shall exhibit a bound for each of the factors on the right.
	
For any $J\subset I$ the set $I\setminus J$ consists of at most two intervals $K$ and $L$, of length $k$ and $l$ respectively. From Lemma~\ref{lm_PRtaZ14} with $h=1$ we find
\begin{equation*}\label{eq_DORLm12}
\int_{K}|h|^{2}d\,q(x)\geq  -C\left(2k^{-1}\|h\|_{L^{2}(K)}^{2}+k\|h'\|_{L^{2}(K)}^{2}\right).
\end{equation*}
Since $\lVert h\rVert_{L^{2}(K)}^{2}\leq kM^{2}$ and $k\leq 1$, this yields
\begin{equation*}
	\int_{K}|h|^{2}d\,q(x)\geq 	-C\left(2M^{2}+\|h'\|_{L^{2}(K)}^{2}\right).
\end{equation*}
Because a similar estimate holds for the interval $L$, we have 
\begin{equation*}
	\int_{I\setminus J}|h|^{2}d\,q(x)\geq -C\left(4M^{2}+\|h'\|_{L^{2}(I)}^{2}\right).
\end{equation*}
Hence, 
\begin{equation}\label{eq_DORLm14}
	\int_{J}|h|^{2}d\,q(x)=\int_{I}|h|^{2}d\,q(x)-\int_{I\setminus J}|h|^{2}d\,q(x)\leq
	C_{1}+C\left(4M^{2}+\|h'\|_{L^{2}(I)}^{2}\right).
\end{equation}
Thus, there exists a bound of the required type for the second factor in \eqref{eq_DORLm10}.
	
On the other hand,
\begin{equation}\label{eq_DORLm16}
\inf_{J}\lVert h\rVert^{2}\leq m^{-2},
\end{equation}
and
\begin{align}\label{eq_DORLm18}
\mathop\mathrm{var}_{I}|h|^{-2} &  =\int_{I}\left|\frac{d}{dx}|h(x)|^{-2}\right|d\,x
=\int_{I}2|h|^{-4}\left|Re(h\overline{h'})\right|d\,x \leq 2m^{-4}\|h\|_{L^{2}(I)}\|h'\|_{L^{2}(I)} \notag \\
& \leq 2m^{-4}M\|h'\|_{L^{2}(I)}.
\end{align}
So in virtue of \eqref{eq_DORLm11}, \eqref{eq_DORLm14}, \eqref{eq_DORLm16} and \eqref{eq_DORLm18} we get a desired bound for $\int_{I}d\,q(x)$.
\end{proof}

Now we are ready to prove Theorem \ref{thM_VtaOR_C}.

Let us consider the operator
\begin{equation*}
\mathrm{B}=\mathrm{S}(q)+(2C^{2}+1)\mathrm{I},
\end{equation*}
where $\mathrm{I}$ is the identity operator with the domain $\mathrm{Dom}(\mathrm{S}(q))$. 
Let us recall that $\mathrm{S}_{0}(q)=\mathrm{S}(q)$. It is obvious that the operator $\mathrm{S}(q)$ has discrete spectrum if and only if the operator $\mathrm{B}$ has. We get 
\begin{equation}\label{eq_54}
 (\mathrm{B} u,u)_{L^{2}(\mathbb{R})}\geq (u,u)_{L^{2}(\mathbb{R})}.
\end{equation}
Then due to Rellich Theorem the operator $\mathrm{B}$ has discrete spectrum if and only if the set
\begin{equation*}
\mathcal{M}=\{u\in \mathrm{Dom}(\mathrm{S}(q))|(\mathrm{B}u,u)_{L^{2}(\mathbb{R})}\leq 1\}
\end{equation*}
is precompact (i. e. every infinite sequence contains a Cauchy-sequence).

The norms of elements of $\mathcal{M}$ are uniformly bounded according to \eqref{eq_54}. 
Hence, choosing $h$ appropriately in \eqref{eq_48}, 
we see that $\|u'\|_{L^{2}(\mathbb{R})}^{2}$ is also uniformly bounded with respect to $u\in\mathcal{M}$. 
Thus, $\mathcal{M}$ is an equicontinuous family of functions $u\in L^{2}(\mathbb{R})$, i. e. 
\begin{equation*}
\|u(x+h)-u(x)\|_{L^{2}(\mathbb{R})}^{2}
\end{equation*}
vanishes uniformly as $h\rightarrow 0$. A compactness theorem of M.~Riesz can now be applied: The set $\mathcal{M}$ is
precompact if and only if
\begin{equation}\label{eq_56}
\lim_{n\rightarrow\infty}\left(\sup_{u\in\mathcal{M}}\int_{x>n}|u|^{2}d\,x\right)=0.
\end{equation}

We shall now prove that the condition 
\begin{equation}\label{eq_52}
\lim_{|a|\rightarrow \infty}\int_{a}^{a+h}d\,q(x)=+\infty \quad\text{ for all }h>0
\end{equation}
is sufficient for the discreteness of the spectrum. 
To this end we suppose that \eqref{eq_56} is not fulfilled. 
This means that we assume the existence of a sequence of functions $u_{n}\in \mathcal{M}$ for which
\begin{equation}\label{eq_58}
\int_{|x|>n}|u_{n}|^{2}d\,x\geq\eta^{-1}>0
\end{equation}
for some $\eta$ independent of $n$. Now
\begin{equation*}
(\mathrm{B}u_{n},u_{n})_{L^{2}(\mathbb{R})}=\int_{\mathbb{R}}|u_{n}'|^{2}d\,x+(2C^{2}+1)
\int_{\mathbb{R}}|u_{n}|^{2}d\,x+\int_{\mathbb{R}}|u_{n}|^{2}d\,q(x)\leq 1,
\end{equation*}
according to \eqref{eq_44}, and if $n\geq 1$, then 
\begin{equation*}
\int_{-n}^{n}|u_{n}'|^{2}d\,x+(2C^{2}+1)\int_{-n}^{n}|u_{n}|^{2}d\,x
+\int_{-n}^{n}|u_{n}|^{2}d\,q(x)\geq 0
\end{equation*}
due to Corollary \ref{cor_10}. Therefore, in view of \eqref{eq_58},
\begin{equation*}
\int_{\mathbb{R}}|u_{n}'|^{2}d\,x+(2C^{2}+1)
\int_{\mathbb{R}}|u_{n}|^{2}d\,x+\int_{\mathbb{R}}|u_{n}|^{2}d\,q(x)\leq
1\leq \eta\int_{x>n}|u_{n}|^{2}d\,x.
\end{equation*}
We split the set $(-\infty, -n)\cup(n,\infty)$ 
into a sum of disjoint intervals $J_{k}$ of equal length $l\leq 1$. (This number $l$ shall be the same
for all $n$. It will be clear below how $l$ is most suitable chosen, depending on the numbers $C$ and $\eta$ only.) Then
\begin{equation}\label{eq_60}
\sum_{k}\left[\int_{J_{k}}|u_{n}'|^{2}d\,x+(2C^{2}+1)
\int_{J_{k}}|u_{n}|^{2}d\,x+\int_{J_{k}}|u_{n}|^{2}d\,q(x)\right]\leq
\eta\sum_{k}\int_{J_{k}}|u_{n}|^{2}d\,x.
\end{equation}
Hence, there exists at least one interval $I_{n}=I$ among $J_{k}$ such that 
\begin{equation}\label{eq_62}
\int_{I}|u_{n}'|^{2}d\,x+(2C^{2}+1)
\int_{I}|u_{n}|^{2}d\,x+\int_{I}|u_{n}|^{2}d\,q(x)\leq
\eta\int_{I}|u_{n}|^{2}d\,x.
\end{equation}
Lemma \ref{lm_PRtaZ14} and \eqref{eq_62} yield
\begin{equation*}
(1-Cl)\|u_{n}'\|_{L^{2}(I)}^{2}+(2C^{2}+1-2Cl^{-1})\|u_{n}\|_{L^{2}(I)}^{2}
\leq \eta\|u_{n}\|_{L^{2}(I)}^{2}.
\end{equation*}
Let $v_{n}$ be a multiple of $u_{n}$ such that $\|v_{n}\|_{L^{2}(I)}^{2}=l$, and let $l<1/C$. Then
\begin{equation*}
\|v_{n}'\|_{L^{2}(I)}^{2}\leq (1-Cl)^{-1}(\eta+2Cl^{-1}-2C^{2}-1)l,
\end{equation*}
which yields 
\begin{equation*}
l\|v_{n}'\|_{L^{2}(I)}^{2}\leq l(1-Cl)^{-1}(\eta l+2C-l(2C^{2}+1)).
\end{equation*}
Since the expression on the right vanishes as $l\rightarrow 0$, there exists a number $l_{0}(\eta,C)$ 
depending only on $\eta$ and $C$ such that $l\leq l_{0}$ implies $l\|v_{n}'\|_{L^{2}(I)}^{2}\leq \frac{1}{2}$. 
Letting the intervals in \eqref{eq_60} have precisely the length $l_{0}$ 
we conclude from the Lemma~\ref{lm_PRtaZ12} that
\begin{equation}\label{eq_64}
1/4\leq |v_{n}(x)|^{2}\leq 9/4
\end{equation}
for all $x$ in $I$. Finally, we conclude from \eqref{eq_62}, which also holds for $v_{n}$ by homogeneity, that 
\begin{equation}\label{eq_66}
\int_{I}|v_{n}(x)|^{2}d\,q(x)\leq \|v_{n}\|_{L^{2}(I)}^{2}(\eta-2C^{2}-1)-\|v_{n}'\|_{L^{2}(I)}^{2}\leq l_{0}(\eta-2C^{2}-1)=K.
\end{equation}
In view of \eqref{eq_64} and \eqref{eq_66}, the assumptions of Lemma~\ref{lm_DOR12} are satisfied. Hence, 
\begin{equation*}
\int_{I}d\,q(x)\leq C_{0},
\end{equation*}
where $C_{0}$ depends only on $\|v_{n}'\|_{L^{2}(I)}^{2}$, $C$, and $K$, i. e. on $\eta$ and $C$ only.

Therefore, if $\mathcal{M}$ is not precompact, we can find a sequence of intervals $I_{n}$ of equal length $l_{0}$ 
and with $I_{n}$ outside the interval $|x|\leq n$ such that $\int_{I_{n}}d\,q(x)$ is uniformly bounded. Then \eqref{eq_52}
cannot be true; hence, $\mathcal{M}$ must be precompact if \eqref{eq_52} holds. This proves the sufficiency assertion of Theorem~\ref{thM_VtaOR_C}.

It remains to prove that condition \eqref{eq_52} for the discreteness of the spectrum is necessary. 
To do this let us consider the operator $\mathrm{B}^{1/2}$ instead of $\mathrm{B}$. 
The operator $\mathrm{B}^{1/2}$ has discrete spectrum if and only if the operator $\mathrm{B}$ has. 
Then Rellich Theorem for $\mathrm{B}^{1/2}$ reads as follows: spectrum of $\mathrm{B}^{1/2}$ is discrete if and only if the set
\begin{equation*}
\mathcal{M'}=\left\{u\in \mathrm{Dom}(\mathrm{B}^{1/2})\left| \|\mathrm{B}^{1/2}u\|_{L^{2}(\mathbb{R})}^{2}+\|u\|_{L^{2}(\mathbb{R})}^{2}\leq 1\right.\right\}
\end{equation*}
is precompact. The operator $\mathrm{B}^{1/2}$ is more convenient than the operator $\mathrm{B}$  for proving the necessity because 
\begin{equation*}
 C_{comp}^{\infty}(\mathbb{R})\subset \mathrm{Dom}(\mathrm{B}^{1/2}).
\end{equation*}
Let notice that $\mathrm{Dom}(\mathrm{B}^{1/2})$ coincides with the domain of the closure of the quadratic form $\dot{t}_{\dot{\mathrm{S}}_{0}(q)}$ generated by the preminimal operator $\dot{\mathrm{S}}_{0}(q)$.

Now, suppose that condition \eqref{eq_56} is not satisfied. 
This is equivalent to the existence of a sequence $\{\Delta\}_{1}^{\infty}$ of disjoint intervals of equal length $\kappa>0$ such that
\begin{equation}\label{eq_68}
\int_{\Delta_{\nu}}d\,q(x)\leq C_{1}
\end{equation}
for all $\nu$. Obviously there is no loss of generality to suppose that $\kappa\leq 1$, 
for otherwise we can find a sequence of intervals contained in $\Delta_{\nu}$ of length $\leq 1$ for which \eqref{eq_68} holds.

We  observe that \eqref{eq_68} implies the existence of an upper bound for the corresponding integral over any sub-interval
$J$ contained in $\Delta_{\nu}$, because 
\begin{equation*}\label{eq_70}
\int_{J}d\,q(x)=\int_{\Delta_{\nu}}d\,q(x)-\int_{\Delta_{\nu}-J}d\,q(x) \leq C_{1}+2C=K
\end{equation*}
in view of $(\mathrm{Br})$.

Let $\varphi_{1}\not\equiv 0$ be a twice continuously differentiable function with support contained in $\Delta_{1}$ and
let $\varphi_{\nu}$ be the translate of $\varphi_{1}$ to the interval $\Delta_{\nu}$. Applying Proposition~\ref{pr_12} we then get
\begin{align}\label{eq_72}
& \int_{\mathbb{R}}|\varphi_{\nu}'|^{2}d\,x+(2C^{2}+1)\int_{\mathbb{R}}|\varphi_{\nu}|^{2}d\,x
 +\int_{\mathbb{R}}|\varphi_{\nu}|^{2}d\,q(x) \\
 & \leq (1+K\kappa)\|\varphi_{\nu}'\|_{L^{2}(\mathbb{R})}^{2}+(2C^{2}+1+2K\kappa^{-1})\|\varphi_{\nu}\|_{L^{2}(\mathbb{R})}^{2} \notag \\
& =(1+K\kappa)\|\varphi_{1}'\|_{L^{2}(\mathbb{R})}^{2}+(2C^{2}+1+2K\kappa^{-1})\|\varphi_{1}\|_{L^{2}(\mathbb{R})}^{2} \notag
\end{align}
for all $\nu$. Since the functions $\varphi_{\nu}$ have disjoint supports, it follows that
\begin{equation*}
\|\varphi_{j}-\varphi_{k}\|_{L^{2}(\mathbb{R})}^{2}=2\|\varphi_{1}\|_{L^{2}(\mathbb{R})}^{2}>0\;\text{when}\; j\neq k.
\end{equation*}
Hence a set containing all the functions $\varphi_{\nu}$ cannot be precompact.

Further, supposing that $\varphi_{1}$ is normed 
so that the right hand side of \eqref{eq_72} does not exceed, say, $\frac{1}{2}$, and using
the fact that $\mathrm{B}^{1/2}\geq \mathrm{I}$, we conclude that the set $\mathcal{M'}$ contains the sequence $\{\varphi_{\nu}\}_{1}^{\infty}$. Therefore $\mathcal{M'}$ is not precompact, and hence the spectrum of $\mathrm{S}(q)$
cannot be discrete. Thus, assumption \eqref{eq_68} must be false if $\mathrm{S}(q)$ has discrete spectrum. 
Consequently, \eqref{eq_52} is a necessary condition. 

The proof of Theorem~\ref{thM_VtaOR_C} is thereby complete.


\section{Proof of Theorem~\ref{thM_VtaOR_D}}

Theorem~\ref{th_PRtaZ10} and Theorem~\ref{th_16} together with Corollary~\ref{cr_dom} prove assertions (I) and (II) of Theorem~\ref{thM_VtaOR_D} respectively.

Let us prove assertion (III) of Theorem~\ref{thM_VtaOR_D}.

We shall deal with the domain of $\mathrm{B}^{1/2}$ 
instead of the domain of the sesquilinear form $t[u,v]$ 
(which is a closure of the form $\dot{t}_{\dot{\mathrm{S}}_{0}(q)}[u,v]$ generated by the preminimal operator $\dot{\mathrm{S}}_{0}(q)$). 
Recall that
\begin{equation*}
\mathrm{B}=\mathrm{S}(q)+(2C^{2}+1)\mathrm{I} \qquad \text{and} \qquad \mathrm{Dom}(\mathrm{B})=\mathrm{Dom}(\mathrm{S}(q)).
\end{equation*} 
The operator $\mathrm{B}$ is selfadjoint and $\mathrm{B}\geq \mathrm{I}$. It is well known that $\mathrm{Dom}(\mathrm{B}^{1/2})$ coincides with $\mathrm{Dom}(t)$.

For arbitrary $f, g \in H_{comp}^{1}(\mathbb{R})$ we define a new inner product 
\begin{equation}\label{eq_74}
\langle f,g \rangle: =\int_{\mathbb{R}}f'\overline{g'}d\,x+\int_{\mathbb{R}}f\overline{g}d\,p(x),
\end{equation}
where $p(x):=q(x)+(2C^{2}+1)x$. Then in view of Corollary~\ref{cr_10LmPRtaZ14} we conclude that 
\begin{equation}\label{eq_76}
\begin{split}
\langle f,f \rangle=\int_{\mathbb{R}}|f'|_{L^{2}(\mathbb{R})}^{2}d\,x&+\int_{\mathbb{R}}|f'|_{L^{2}(\mathbb{R})}^{2}d\,p(x)\\
&\geq(1-C h)\|f'\|_{L^{2}(\mathbb{R})}^{2}+(2C^{2}+1-2Ch^{-1})\|f\|_{L^{2}(\mathbb{R})}^{2}
\end{split}
\end{equation}
for all positive $h\leq 1$. Therefore with a proper choice of $h$ we get 
\begin{equation}\label{eq_78}
\langle f,f \rangle\geq C_{1}(\|f'\|_{L^{2}(\mathbb{R})}^{2}+\|f\|_{L^{2}(\mathbb{R})}^{2})
\end{equation}
for some positive constant $C_{1}$. Closing $H_{comp}^{1}(\mathbb{R})$ in the norm \eqref{eq_76} we get a Hilbert space $\mathcal{R}$.
\begin{lemmaEng}\label{l_22}
The embedding $\mathcal{R}\subset H^{1}(\mathbb{R})$ holds true and the inner product in $\mathcal{R}$ is given by
\begin{equation}\label{eq_80}
	\langle f,h \rangle=\int_{\mathbb{R}}f'\overline{h'}d\,x+\int_{\mathbb{R}}f\overline{h}d\,p
\end{equation}
for any $h\in H_{comp}^{1}(\mathbb{R})$ and $f\in\mathcal{R}$.
\end{lemmaEng}
\begin{proof}
The first assertion of the lemma follows immediately from \eqref{eq_78}. To prove the second one, let $f$ be defined by a sequence $\{f_{\nu}\}_{1}^{\infty}$ of elements from $H_{comp}^{1}(\mathbb{R})$. Then
\begin{equation}\label{eq_82}
	\langle	f_{\nu},h\rangle=\int_{\mathbb{R}}f_{\nu}'\overline{h'}d\,x+\int_{\mathbb{R}}f_{\nu}\overline{h}d\,p
\end{equation}
by definition. But $f_{\nu}'$ and $f_{\nu}$ converge in $L_{2}(\mathbb{R})$ to $f'$ and $f$ respectively. 
Hence, $f_{\nu}$ converges uniformly to $f$ on the support of $h$. 
Thus, the integral in \eqref{eq_82} tends to the integral in \eqref{eq_80}, which proves the lemma.
\end{proof}
\begin{lemmaEng}\label{l_24}
The domain $\mathrm{Dom}(\dot{\mathrm{S}}_{0}(q))$ is dense in $\mathcal{R}$.
\end{lemmaEng}
\begin{proof}
Suppose that $\langle f,u\rangle=0$ for every $u\in\mathrm{Dom}(\dot{\mathrm{S}}_{0}(q))$ and some $f\in\mathcal{R}$. Integrating by
parts we obtain 
\begin{equation*}
	0=\langle f,u \rangle=\int_{\mathbb{R}}f'\overline{u'}d\,x+	\int_{\mathbb{R}}f\overline{u}d\,p 
	=-\int_{\mathbb{R}}f\overline{u''}d\,x+\int_{\mathbb{R}}f\overline{u}d\,p=(f,\mathrm{B} u)_{L^{2}(\mathbb{R})},
\end{equation*}
according to lemma \ref{l_22}. 

But $\mathrm{B}(\mathrm{Dom}(\dot{\mathrm{S}}_{0}(q)))$ is dense in $L_{2}(\mathbb{R})$, 
hence $f=0$, which proves the lemma.
\end{proof}
\begin{theoremEng}\label{th_20}
The domain of the operator $\mathrm{B}^{1/2}$ coincides with $\mathcal{R}$.
\end{theoremEng}
\begin{proof}
We first note that $\mathrm{Dom}(\dot{\mathrm{S}}_{0}(q))$ is dense 
in $\mathrm{Dom}(\mathrm{S}(q))$ in the graph norm, because $\mathrm{S}$ is the closure of its restriction to $\mathrm{Dom}(\dot{\mathrm{S}}_{0}(q))$. 
Using well-known functional calculus for operators, we then conclude that $\mathrm{Dom}(\dot{\mathrm{S}}_{0}(q))$ is also dense in the domain of $\mathrm{B}^{1/2}$ in the corresponding graph norm. 
Since
\begin{equation*}
\left(\mathrm{B}^{1/2}u,\mathrm{B}^{1/2}u\right)_{L^{2}(\mathbb{R})}=\left(\mathrm{B} u,u\right)_{L^{2}(\mathbb{R})}=\langle u,u \rangle
\end{equation*}
for all $u\in \mathrm{Dom}(\dot{\mathrm{S}}_{0}(q))$, 
then the domain of $\mathrm{B}^{1/2}$ is obtained by closing $\mathrm{Dom}(\dot{\mathrm{S}}_{0}(q))$ 
with respect to the norm in $\mathcal{R}$. Thus	
\begin{equation*}
	\mathrm{Dom}(\mathrm{B}^{1/2})\subset\mathcal{R}.
\end{equation*}
But Lemma \ref{l_24} shows that $\mathrm{Dom}(\mathrm{B}^{1/2})$ cannot be a proper subset of $\mathcal{R}$, 
for then some $f\in\mathcal{R}\setminus\{0\}$ would be orthogonal to all $h\in\mathrm{Dom}(\mathrm{B}^{1/2})$ 
and hence to all $u\in \mathrm{Dom}(\dot{\mathrm{S}}_{0}(q))$, which is possible only for $f=0$. Thus
\begin{equation*}
	\mathrm{Dom}(\mathrm{B}^{1/2})=\mathcal{R}
\end{equation*}
and the theorem is proved.
\end{proof}

Remark that we have not given any explicit form for the inner product $\langle f,g \rangle$ 
of arbitrary elements in $\mathcal{R}$. 
It may be of interest to note, however, that an integral expression corresponding to \eqref{eq_74} does give the
inner product $\langle f,g \rangle$ for arbitrary $f, g\in\mathcal{R}$.
\begin{lemmaEng}\label{l_26}
The inner product in $\mathcal{R}$ is given by
\begin{equation*}\label{eq_84}
	\langle f,g\rangle=\lim_{M,N\rightarrow\infty}\left(\int_{-M}^{N}f'\overline{g'}d\,x+\int_{-M}^{N}f\overline{g}d\,p\right).
\end{equation*}
\end{lemmaEng}
\begin{proof}
It is sufficient to prove that for every $f\in\mathcal{R}$
\begin{equation}\label{eq_86}
	\langle f,f\rangle=\lim_{M,N\rightarrow\infty}\left(\int_{-M}^{N}|f'|^{2}d\,x+
	\int_{-M}^{N}|f|^{2}d\,p\right),
\end{equation}
because then
\begin{align*}
	4\langle f,g\rangle & =\langle f+g,f+g\rangle-\langle f-g,f-g\rangle+i\left[\langle f+i g,f+i g\rangle-\langle f-i g,f-i g\rangle\right] \\
	& =\lim_{M,N\rightarrow\infty}4\left(\int_{-M}^{N}f'\overline{g'}d\,x+
	\int_{-M}^{N}f\overline{g}d\,p\right).
\end{align*}
	
We define
\begin{equation*}
	\langle f,f\rangle_{n}=\int_{n-1}^{n}|f'|^{2}d\,x+\int_{n-1}^{n}|f|^{2}d\,p
\end{equation*}
for any $f\in\mathcal{R}$ and infer from the Corollary~\ref{cor_10} to Lemma \ref{lm_PRtaZ14} that the number $\langle f,f\rangle_{n}$ is non-negative for all $n$. We proceed to prove that the series
\begin{equation}\label{eq_88}
	P(f)=\sum_{n=-\infty}^{\infty}\langle f,f\rangle_{n}
\end{equation}
converges to $\langle f,f\rangle$.
	
For any $h\in H_{comp}^{1}(\mathbb{R})$ the series in \eqref{eq_88} is finite and $P(h)=\langle h,h\rangle$. Now, let $f$ be an arbitrary element in $\mathcal{R}$, defined by a 	Cauchy-sequence $\{f_{\nu}\}_{1}^{\infty}$ of elements in $H_{comp}^{1}(\mathbb{R})$. Then, as we have seen, $f_{\nu}'$ converges in $L_{2}(\mathbb{R})$ to $f'$ and $f_{\nu}$ converges 	uniformly to $f$ on compacts. Thus the individual terms $\langle f_{\nu},f_{\nu}\rangle_{n}$ converge to $\langle f,f\rangle_{n}$ for every $n$. But $\langle f_{\nu},f_{\nu}\rangle$ converges to $\langle f,f\rangle$ and hence Fatou's lemma shows that
\begin{align*}
	P(f) & =\sum_{n=-\infty}^{\infty}\langle f,f\rangle_{n}=
	\sum_{n=-\infty}^{\infty}\lim_{\nu\rightarrow\infty}\langle f_{\nu},f_{\nu}\rangle_{n}
	\leq \lim_{\nu\rightarrow\infty}\sum_{n=-\infty}^{\infty}\langle f_{\nu},f_{\nu}\rangle_{n}\\
	& =\lim_{\nu\rightarrow\infty}P(f_{\nu})=\lim_{\nu\rightarrow\infty}\langle f_{\nu},f_{\nu}\rangle
	=\langle f,f\rangle.
\end{align*}
Thus, the series $P(f)$ converges, because its terms are non-negative, and $P(f)\leq \langle f,f\rangle$.
	
To obtain the converse inequality, we define $\langle f,h\rangle_{n}$ for $f\in\mathcal{R}$ and $h\in H_{comp}^{1}(\mathbb{R})$ by
\begin{equation*}
	\langle f,h\rangle_{n}=\int_{n-1}^{n}f'\overline{h'}d\,x+\int_{n-1}^{n}f\overline{h}d\,p.
\end{equation*}
Lemma \ref{l_22} shows that
\begin{equation*}
	\langle f,h\rangle=\sum_{n=-\infty}^{\infty}\langle f,h\rangle_{n},
\end{equation*}
the series in fact being finite. Since $\langle f,f\rangle_{n}$ is positive definite we get by Schwarz' inequality
\begin{equation*}
	|\langle f,h\rangle_{n}|^{2}\leq \langle f,f\rangle_{n}\langle h,h\rangle_{n}.
\end{equation*}
Hence
\begin{equation*}
	|\langle f,h\rangle|^{2}=\left|\sum_{n=-\infty}^{\infty}\langle f,h\rangle_{n}\right|^{2}
	\leq \sum_{n=-\infty}^{\infty}\langle f,f\rangle_{n}\sum_{n=-\infty}^{\infty}\langle h,h\rangle_{n}
	=P(f)\langle h,h\rangle.
\end{equation*}
This proves $P(f)\geq \langle f,f\rangle$, as $H_{comp}^{1}(\mathbb{R})$ is dense in $\mathcal{R}$. 
Therefore, $P(f)=\langle f,f\rangle$ in view of the inequality obtained above.
	
We have thus proved that the integral in \eqref{eq_86} converges to $\langle f,f\rangle$ 
when $\mathbb{Z}\ni M, N \rightarrow\infty$. 
But $f$ and $f'$ are both in $L^{2}(\mathbb{R})$; 
therefore we can apply Lemma~\ref{lm_PRtaZ14} to arbitrary $M$ and $N$ (as in the proof of Theorem~\ref{th_16}) to obtain
\begin{align*}
	\int_{-[M]-1}^{[N]+1}|f'|^{2}d\,x+\int_{-[M]-1}^{[N]+1}|f|^{2}d\,p+o(1)
	& \geq\int_{-M}^{N}|f'|^{2}d\,x+\int_{-M}^{N}|f|^{2}d\,p \\
	& \geq\int_{-[M]}^{[N]}|f'|^{2}d\,x+\int_{-[M]}^{[N]}|f|^{2}d\,p-o(1),
\end{align*}
with $[N]$ denoting the greatest integer $\leq N$. But we have proved that 
the expressions on the left and on the right both tend to $\langle f,f\rangle$. 
Hence the lemma is proved.
\end{proof}
\begin{theoremEng}\label{th_22}The equality 
\begin{equation*}
 \mathrm{Dom}(\mathrm{B}^{1/2})=\left\{u\in	H^{1}(\mathbb{R})\left|\;\exists\int_{\mathbb{R}}|u|^{2}d\,q\in \mathbb{R}\right.\right\},
\end{equation*}
holds, where the integral $\int_{\mathbb{R}}u\overline{v}d\,q(x)$ is considered as improper Riemann--Stieltjes integral.
\end{theoremEng}
\begin{proof}
We have just shown that the limit in \eqref{eq_86} exists and is finite for all $f\in\mathcal{R}$. 
Since $f$ and $f'$ are in $L^{2}(\mathbb{R})$, then the potential energy exists and is finite.
	
Conversely, if $f$ satisfies the conditions of the theorem, the formula
\begin{equation*}
	F(g)=\lim_{M,N\rightarrow\infty}\left(\int_{-M}^{N}g'\overline{f'}d\,x+
	\int_{-M}^{N}g\overline{f}d\,p\right)
\end{equation*}
defines a continuous functional realized by some element $h\in\mathcal{R}$, and it is not difficult to prove that
the function $f-h$ must then be an $L^{2}$-solution to the 	equation $B u=0$. 
Since $B$ is positive this implies $f=h$, hence $f\in\mathcal{R}$ and the theorem is proved.
\end{proof}

\section{Some remarks}
Standard arguments show that the minimal operator $\mathrm{S}_{0}(q)$ is bounded below in the Hilbert space $L^{2}(\mathbb{R})$ if and only if minimal operators $\mathrm{S}_{0}^{\pm}(q)$, generated by the differential expression $\mathrm{S}(q)$ in Hilbert spaces $L^{2}(\mathbb{R}_{\pm})$ correspondingly are bounded below. 
Herein the discreteness of the spectrum of operator $\mathrm{S}_{0}(q)$ is equivalent to the discreteness of the both spectra of the operators $\mathrm{S}_{D}^{\pm}(q)$ that correspond to the selfadjoint extensions of operators $\mathrm{S}_{0}^{\pm}(q)$ with homogeneous Dirichlet condition at the end of the semi-axis $\mathbb{R}_{\pm}$.
Therefore Theorems \ref{thM_VtaOR_B} and \ref{thM_VtaOR_C} (reformulated accordingly) also hold for the Schr\"{o}dinger operators on the semi-axis, which were studied in~\cite{AlKsMl2010}. These theorems generalize the results \cite[Lemma~III.1]{AlKsMl2010} and \cite[Theorem~IV.1]{AlKsMl2010}.

The following example illustrates the difference between our results and the former ones.

\textbf{Example.} Let $\{x_{n}\}_{n=1}^{\infty}$ be an arbitrary strictly increasing unbounded sequence of positive numbers such that $x_{n+1}-x_{n}\rightarrow 0$ as $n\rightarrow \infty$. 
Choose $\rho>0$ and $\{\alpha_{2n-1}\}_{n=1}^{\infty}\subset \mathbb{R}_{+}$ arbitrarily.  
Consider the potential of the form 
\begin{equation*}
 q'(x)=\sum_{n=1}^{\infty}(\rho+\alpha_{2n-1})\delta(x-x_{2n-1})-\sum_{n=1}^{\infty}\rho\delta(x-x_{2n}).
\end{equation*}
Simple verification shows that the Radon measure $q'(x)$ does not satisfy conditions (A) and (B) from paper  \cite{Bra1985} and conditions of Theorem~IV.1 from \cite{AlKsMl2010}. 
However, $q'(x)$ satisfies condition $(\mathrm{Br})$. 
Therefore, operator $\mathrm{S}_{D}^{+}(q)$ is bounded below and self-adjoint. 
Due to Theorem \ref{thM_VtaOR_C} its spectrum is discrete if and only if 
\begin{equation*}
 \sum_{x_{2n-1}\in \Delta}\alpha_{2n-1}\rightarrow +\infty,
\end{equation*}
where the interval $\Delta\subset\mathbb{R}_{+}$ moves to $+\infty$ preserving its length.

\section*{Appendix}
Let us formulate some known statements about the operators $\dot{\mathrm{S}}_{0}(q)$, $\mathrm{S}_{0}(q)$ and  $\mathrm{S}(q)$, which are used in the paper. 
Their proofs may be found in \cite{MkMlMFAT2013n1, MkMlMFAT2013n2, ET2013, EGNT2013}.
\begin{propositionEng*}\label{pr_VtaOR10}
The operators $\dot{\mathrm{S}}_{0}(q)$, $\mathrm{S}_{0}(q)$, and $\mathrm{S}(q)$ have the following properties:
	\begin{itemize}
		\item[$1^{0}$.] The domain $\mathrm{Dom}(\dot{\mathrm{S}}_{0}(q))$ of the preminimal operator $\dot{\mathrm{S}}_{0}(q)$ is dense in the Hilbert space $L^{2}(\mathbb{R})$.
		\item[$2^{0}$.] The operator $\dot{\mathrm{S}}_{0}(q)$ is symmetric and therefore it is closable.
		\item[$3^{0}$.] Let $\mathrm{S}_{0}(q):=\left(\dot{\mathrm{S}}_{0}(q)\right)^{\sim}$. Then
		\begin{equation*}\label{eq_VtaOR30}
		\left(\dot{\mathrm{S}}_{0}(q)\right)^{\ast}=\mathrm{S}(q)\qquad\text{and}\qquad \dot{\mathrm{S}}_{0}(q)\subset \mathrm{S}_{0}(q)\subset \mathrm{S}(q).
		\end{equation*}
		\item[$4^{0}$.] The minimal operator $\mathrm{S}_{0}(q)$ is a densely defined, closed, and symmetric operator with the deficiency index $(d,d)$, where $0\leq d\leq 2$. The operators $\mathrm{S}_{0}(q)$ and $\mathrm{S}(q)$ are mutually adjoint, i. e. 
		\begin{equation*}\label{eq_VtaOR32}
		\mathrm{S}_{0}^{\ast}(q)=\mathrm{S}(q)\quad\text{and}\quad  \mathrm{S}^{\ast}(q)=\mathrm{S}_{0}(q).
		\end{equation*}
		\item[$5^{0}$.] The domain $\mathrm{Dom}(\mathrm{S}_{0}(q))$ of the minimal operator $\mathrm{S}_{0}(q)$ has the form:
		\begin{equation*}
		\mathrm{Dom}(\mathrm{S}_{0}(q))=\left\{u\in \mathrm{Dom}(\mathrm{S}(q))\left| [u,v]_{+\infty}-[u,v]_{-\infty}=0\quad \forall v\in \mathrm{Dom}(\mathrm{S}(q))\right.\right\},
		\end{equation*}
		where $[u,v]\equiv [u,v](x):=u(x)\overline{v^{[1]}(x)}-u^{[1]}(x)\overline{v(x)}$.
		\item[$6^{0}$.] The domains of the operators $\dot{\mathrm{S}}_{0}(q)$, $\mathrm{S}_{0}(q)$ and $\mathrm{S}(q)$ satisfy the embeddings:
		\begin{align*}
		\mathrm{Dom}(\dot{\mathrm{S}}_{0}(q)) & \subset H_{comp}^{1}(\mathbb{R}), \\
		\mathrm{Dom}(\mathrm{S}_{0}(q)) & \subset H_{loc}^{1}(\mathbb{R})\cap L^{2}(\mathbb{R}), \\
		\mathrm{Dom}(\mathrm{S}(q)) & \subset H_{loc}^{1}(\mathbb{R})\cap L^{2}(\mathbb{R}).
		\end{align*}
	\end{itemize}
\end{propositionEng*}



\begin{thebibliography}{99}
\bibitem{AlKsMl2010}
 {S. Albeverio, A. Kostenko, M. Malamud},
 {\textit{Spectral theory of semibounded Sturm--Liouville operators with local point interactions on a discrete set}},
 {J. Math. Phys.}
 {\textbf{51}}
 {(2010)},
 {24~pp.}

\bibitem{Bra1985}
 {J. Brasche},
 {\textit{Perturbation of Schr\"{o}dinger Hamiltonians by measures --- selfadjointness and semiboundedness}},
 {J. Math. Phys.}
 {\textbf{26}}
 {(1985)},
 {no.~4},
 621--626.

\bibitem{Bri1959}
 {I. Brinck},
 {\textit{Self-adjointness and spectra of Sturm-Liouville operators}},
 {Math. Scand.}
 {\textbf{7}}
 {(1959)},
 {no.~1},
 {219--239}.

\bibitem{Gn1956}
 {T. Ganelius},
 {\textit{Un th\'{e}or\`{e}me taub\'{e}rien pour la transformation de Laplace}},
 {C. R. Acad. Sci. Paris}
 {\textbf{242}}
 {(1956)},
 {719--721}.

\bibitem{EGNT2013}
 {J. Eckhardt, F. Gesztesy, N. Nichols, G. Teschl},
 {\textit{Weyl--Tichmarsh theory for Sturm--Liouville operators with distributional potentials}},
 {Opuscula Math.}
 {\textbf{33}}
 {(2013)},
 {no.~3},
 {467--563}.
 
\bibitem{ET2013}
 {J. Eckhardt, G. Teschl},
 {\textit{Sturm-Liouville operators with measure-valued coefficients}},
 {J. Anal. Math.}
 {\textbf{120}}
 {(2013)},
 {151--224}.

\bibitem{Kt1995}
 {T.~Kato},
 {\textit{Perturbation theory for linear operators}},
 {Springer},
 {Berlin, etc.},
 {1995}.

\bibitem{KshDdk2016}
 {V. Koshmanenko, M. Dudkin},
 {\textit{The method of rigged spaces in singular perturbation theory of self-adjoint operators}},
 {Operator Theory: Advances and Applications}
 {\textbf{253}},
 {Birkh\"{a}user/Springer},
 {Cham},
 {2016}.

\bibitem{MkMlMFAT2013n1}
 {V. Mikhailets, V. Molyboga},
 {\textit{Schr\"{o}dinger operators with complex singular potentials}},
 {Methods Funct. Anal. Topology} 
 {\textbf{19}}
 {(2013)},
 {no.~1},
 {16--28}. 
 
\bibitem{MkMlMFAT2013n2}
 {V. Mikhailets, V. Molyboga},
 {\textit{Remarks on Schr\"{o}dinger operators with singular matrix potentials}},
 {Methods Funct. Anal. Topology} 
 {\textbf{19}}
 {(2013)},
 {no.~2},
 {161--167}. 
 
\bibitem{MkMrNvMFAT2017}
 {V. Mikhailets, A. Murach, V. Novikov},
 {\textit{Localization principles for Schr\"{o}dinger operator with a singular matrix potential}},
 {Methods Funct. Anal. Topology} 
 {\textbf{23}}
 {(2017)},
 {no.~4},
 {367--377}. 


\bibitem{Mol1953}
 {A. Mol\v{c}anov},
 {\textit{On conditions for discreteness of the spectrum of self-adjoint differential equations of the second order}},
 {Trudy Moskov. Mat. Ob\v{s}\v{c}}
 {\textbf{2}}
 {(1953)},
 {169--199}.
 {(Russian)}
 

\bibitem{SvSh2003}
 {A. Savchuk, A. Shkalikov},
 {\textit{Sturm--Liouville operators with distribution potentials}}
 {(Russian)},
 {Tr. Mosk. Mat. Obs.}
 {\textbf{64}}
 {(2003)},
 {159--212};
 {translation in Trans. Moscow Math. Soc. (2003), 143--192}.



\end{thebibliography}
\end{document}